\documentclass[12pt]{tran-l}

\newtheorem{thm}{Theorem}[subsection]
\newtheorem{cor}[thm]{Corollary}
\newtheorem{lem}[thm]{Lemma}

\theoremstyle{definition}

\theoremstyle{remark}
\newtheorem{rem}[thm]{Remark}

\usepackage[utf8]{inputenc}   
\usepackage[T1]{fontenc}      
\usepackage[english]{babel}   
\usepackage[a4paper, left=3cm, right=2.5cm, top=2.5cm, bottom=2.5cm]{geometry}

\numberwithin{equation}{subsection}

\begin{document}

\title[Matrix Representations of Jacobsthal Numbers]
{Matrix Representations and Arithmetic Properties of Jacobsthal Numbers via Binary 3x3 Matrices}

\author{Wilson Arley Martinez * , Samin Ingrith Ceron}

\address{Martinez, W.A.; Departmento de Matem\'aticas, Universidad del Cauca , Popay\'an , Colombia}
\email{wamartinez@unicauca.edu.co}

\address{Ceron, S.I.; Departmento de Matem\'aticas, Universidad del Cauca , Popay\'an , Colombia}
\email{sicbravo@gmail.com}

\thanks{This work was completed with the support of the Universidad del Cauca.}

\thanks{The author was also supported by the research group “Estructuras Algebraicas, Divulgaci\'on Matem\'atica y Teorías Asociadas. @DiTa”.}

\thanks{* Corresponding Author: Wilson Arley Martinez.}




\date{April 26, 2026; revised April 26, 2026}

\dedicatory{}

\commby{W.A.M}


\begin{abstract}
We study matrix representations of the Jacobsthal sequence generated by binary $3\times 3$ matrices with determinants $0$, $2$, and $-2$, using linear algebraic methods analogous to Fibonacci-type constructions. Explicit formulas for matrix powers are obtained, yielding identities for Jacobsthal numbers, including convolution formulas, trace relations, determinant expressions, and Cassini-type identities.  We further derive congruence relations and recurrence formulas, and analyze arithmetic properties such as partial sums and the Sidon-type structure of the sequence. Finally, we prove that exactly three conjugacy classes of binary $3\times 3$ matrices generate the Jacobsthal sequence, providing a unified algebraic framework that connects matrix theory with second-order linear recurrences.
\end{abstract}

\maketitle

\section*{Introduction}

The \emph{Jacobsthal sequence} $(J_n)_{n\ge 0}$ is defined by the second-order linear recurrence
\[
J_{n+1} = J_n + 2J_{n-1}, \quad n \ge 1, \qquad J_0 = 0, \; J_1 = 1.
\]
The first terms are
\[
0, 1, 1, 3, 5, 11, 21, 43, 85, \dots
\]
As a member of the family of Fibonacci-type sequences, the Jacobsthal sequence exhibits a rich algebraic structure and has been studied in connection with number theory, combinatorics, coding theory, and discrete dynamical systems.

\quad

Matrix methods provide an efficient framework for analyzing linear recurrences. It is well known (see, e.g., \cite{KokenBozkurt2008, KokenBozkurt2008b}) that the Jacobsthal numbers can be generated through powers of the companion matrix
\[
A =
\begin{pmatrix}
1 & 2 \\
1 & 0
\end{pmatrix}.
\]
For every integer $n \ge 1$, one has
\[
A^n =
\begin{pmatrix}
J_{n+1} & 2J_n \\
J_n & 2J_{n-1}
\end{pmatrix},
\]
which provides a concise linear algebraic representation of the sequence and facilitates the derivation of structural properties.

\quad

The arithmetic properties of Jacobsthal numbers have been extensively studied in the literature; see, for example, \cite{Horadam1996, KokenBozkurt2008b} and subsequent developments in the theory of second-order linear recurrences. Many classical identities satisfied by Jacobsthal numbers arise as special cases of general identities for Lucas sequences (see, e.g., \cite{Adegoke2019}). In particular, the following well-known identities hold:

\begin{equation}\label{eq:jacobsthal-identities}
\begin{aligned}
J_{n+1}J_{n-1} - J_n^2 &= (-1)^n 2^{n-1}, \\
J_{2n} &= J_n \bigl(J_{n+1} + 2J_{n-1}\bigr), \\
\sum_{k=2}^{n} J_k &= \frac{1}{2}\bigl(J_{n+2} - 3\bigr).
\end{aligned}
\end{equation}
These identities are classical and appear explicitly in the literature on Jacobsthal and related sequences. They illustrate structural properties analogous to Cassini-type relations, duplication formulas, and summation identities, highlighting the interplay between algebraic methods and linear recurrence theory.

\quad

Motivated by these developments, we investigate matrix representations of the Jacobsthal sequence arising from binary \(3\times 3\) matrices with determinants \(0\), \(2\), and \(-2\). For each class, we derive explicit formulas for \(U^n\), expressing their entries in terms of Jacobsthal numbers and thereby obtaining a unified framework for the derivation of identities.

\quad

Within this framework, we establish convolution identities, trace and determinant formulas, and inversion relations. In particular, we recover classical results such as the Binet formula and Cassini-type identities, and further obtain refined relations involving parity corrections, mixed recurrences, congruences modulo powers of \(2\), partial sum formulas, and the Sidon-type property of the sequence.

\quad

Finally, we address a structural classification problem. By performing a complete enumeration of binary $3\times 3$ matrices and classifying them up to conjugation, we show that exactly three conjugacy classes generate the Jacobsthal sequence. Representatives of these classes correspond precisely to the matrix families studied in this work. This classification provides a conceptual framework that connects matrix theory, recurrence relations, and computational algebra.

\section{Recurrences Generated by Binary $3 \times 3$ Matrices with Determinant Zero}
Consider the generating matrix
\[
u = \begin{pmatrix}0 & 0 & 1\\ 0 & 0 & 1\\ 1 & 1 & 1\end{pmatrix}.
\]
In this section, we examine the Jacobsthal  \textit{u}-matrix with determinant~0, which provides a matrix representation of Jacobsthal  numbers. This matrix is employed to compute powers \(u^n\) and derive associated identities; to determine the characteristic roots and the Binet formulas for both the Jacobsthal  sequence and its generalized form; to establish identities involving these sequences; and to obtain summation formulas for the Jacobsthal  numbers.

\subsection{The Matrix Representation of Jacobsthal Numbers}
\begin{lem}\label{l:1} 
Let $u$ be the matrix
\[
u = \begin{pmatrix} 
                    0   &  0   &  1 \\
                    0   &  0   &  1 \\
                    1   &  1   &  1\\
\end{pmatrix}.
\]  
Then
\[
u^{n} =  \begin{pmatrix} \vspace{0.2cm}
                   J_{n-1}  & J_{n-1}   & J_{n}\\ \vspace{0.2cm}
                   J_{n-1}  & J_{n-1}   & J_{n} \\
                   J_{n}    & J_{n}     & J_{n+1} \\
\end{pmatrix}  
\quad \text{and} \quad
\operatorname{Tr}(u^n) = 2 J_{n-1} + J_{n+1},
\]
where \( n\in \mathbb{Z}_{>0} \) and \( J_{n} \) denotes the Jacobsthal numbers defined by the recurrence relation  
\[
J_n = J_{n-1} + 2J_{n-2},
\]
with initial conditions \( J_0 = 0 \) and \( J_1 = 1 \).
\end{lem}

\begin{proof}
We use the principle of mathematical induction (PMI). When $n = 2$,
\[
u^{2} =
\begin{bmatrix}  
                   J_{1}  & J_{1}   & J_{2}\\ \vspace{0.2cm}
                   J_{1}  & J_{1}   & J_{2} \\
                   J_{2}  & J_{2}   & J_{3} \\ 
\end{bmatrix}
=
\begin{bmatrix} 
1 & 1 & 1\\ 
1 & 1 & 1\\  
1 & 1 & 3\\ 
\end{bmatrix},
\]
so the result holds.

\quad

Assume that the statement holds for a positive integer $n=k$, that is,
\[
u^{k} =  \begin{pmatrix} \vspace{0.2cm}
                   J_{k-1}  & J_{k-1}   & J_{k}\\ \vspace{0.2cm}
                   J_{k-1}  & J_{k-1}   & J_{k} \\
                   J_{k}    & J_{k}     & J_{k+1} \\
\end{pmatrix}.
\]
Now we show that it also holds for $n=k+1$. We have
\[
u^{k+1} = u^k u 
= \begin{pmatrix} \vspace{0.2cm}
                   J_{k-1}  & J_{k-1}   & J_{k}\\ \vspace{0.2cm}
                   J_{k-1}  & J_{k-1}   & J_{k} \\
                   J_{k}    & J_{k}     & J_{k+1} \\
\end{pmatrix}
\begin{pmatrix} 
                    0   &  0   &  1 \\
                    0   &  0   &  1 \\
                    1   &  1   &  1\\
\end{pmatrix}.
\]
Hence
\[
= 
\begin{pmatrix} \vspace{0.2cm}
       J_{k}    & J_{k}    & J_{k} + 2J_{k-1}  \\ \vspace{0.2cm}
       J_{k}    & J_{k}    & J_{k} + 2J_{k-1}  \\
       J_{k+1}  & J_{k+1}  & J_{k+1} + 2J_{k} \\
\end{pmatrix}.
\]
Using the recurrence relation \(J_{k+1} = J_k + 2J_{k-1}\), we obtain
\[
u^{k+1} =
\begin{pmatrix} \vspace{0.2cm}
       J_{k}    & J_{k}    & J_{k+1}  \\ \vspace{0.2cm}
       J_{k}    & J_{k}    & J_{k+1}  \\
       J_{k+1}  & J_{k+1}  & J_{k+2} \\
\end{pmatrix}.
\]
Therefore, the result holds for $k+1$, completing the proof.

\end{proof}

Analogously, the following lemma can be proved by mathematical induction.

\begin{lem}\label{l:2} 
Let $u$ be the matrix   
\[
u = \begin{pmatrix} 
                    0   &  1   &  0 \\
                    1   &  1   &  1 \\
                    0   &  1   &  0\\
\end{pmatrix}.
\]  
Then
\[
u^{n} =  \begin{pmatrix} \vspace{0.2cm}
 J_{n-1}  &  J_{n}  &  J_{n-1} \\  \vspace{0.2cm}
 J_{n}    &  J_{n+1}  &  J_{n}  \\
 J_{n-1}  &  J_{n}  &  J_{n-1} \\
\end{pmatrix}
\]
where \( n\in \mathbb{Z}_{>0} \) and \( J_{n} \) denotes the Jacobsthal numbers, defined by the recurrence relation  
\[
J_n = J_{n-1} + 2J_{n-2},
\]
with initial conditions \( J_0 = 0 \) and \( J_1 = 1 \).
\end{lem}

\begin{cor}

The following matrices are similar:
\[
\begin{pmatrix} 
                    0   &  0   &  1 \\
                    0   &  0   &  1 \\
                    1   &  1   &  1\\
\end{pmatrix}, 
\begin{pmatrix} 
                    0   &  1   &  0 \\
                    1   &  1   &  1 \\
                    0   &  1   &  0\\
\end{pmatrix}.
\]
Consequently, for every \( n \in \mathbb{N} \),
\[
\operatorname{Tr}(u^n) = J_{n} + 4J_{n-1}.
\]

\end{cor}

Formulas for $J_{m+n}$ are already known in the literature (see, e.g., \cite{Horadam1996,KokenBozkurt2008}).  The identities presented here are algebraically equivalent to those results,  but are written in a form that arises naturally from our matrix representation. 

\begin{cor}\label{c:24}
For $n\geq 1$ and $m\geq 1$, we have the following identities:
\begin{align}
 J_{m+n-1} &= J_{m}J_{n} + 2J_{m-1}J_{n-1},\\
 J_{m+n}   &= J_{m+1}J_{n} + 2 J_{m}J_{n-1}. \label{i:24}
\end{align}
\end{cor}

\begin{proof}

For \( m, n \geq 1 \), we know that \( u^{m+n} = u^m u^n \). Since we have defined \( u^n \) as follows, the corresponding expressions in matrix form are:

\[
u^m =
\begin{pmatrix} \vspace{0.2cm}
J_{m-1} & J_{m-1} & J_{m}\\ \vspace{0.2cm}
J_{m-1} & J_{m-1} & J_{m}\\
J_{m}   & J_{m}   & J_{m+1}
\end{pmatrix},  \quad
u^{n} =
\begin{pmatrix} \vspace{0.2cm}
J_{n-1} & J_{n-1} & J_{n}\\ \vspace{0.2cm}
J_{n-1} & J_{n-1} & J_{n}\\
J_{n}   & J_{n}   & J_{n+1}
\end{pmatrix}.
\]
The product \( u^m u^n \) is represented by the following \( 3 \times 3 \) matrix:

\[
u^m u^n =
\begin{bmatrix}
A & A & C \\
A & A & C \\
C & C & B
\end{bmatrix}
\]
where

\[
\begin{aligned}
A &= 2 J_{m-1}J_{n-1} + J_{m}J_{n},\\
C &= 2 J_{m}J_{n-1} + J_{m+1}J_{n},\\
B &= 2 J_{m}J_{n} + J_{m+1}J_{n+1}.
\end{aligned}
\]
On the other hand, we have

\[
u^{m+n} =
\begin{pmatrix} \vspace{0.2cm}
J_{m+n-1} & J_{m+n-1} & J_{m+n}\\ \vspace{0.2cm}
J_{m+n-1} & J_{m+n-1} & J_{m+n}\\
J_{m+n}   & J_{m+n}   & J_{m+n+1}
\end{pmatrix}.
\]
Equating the two matrices obtained from the matrix multiplication yields the identities stated in the corollary.

\end{proof}

\subsection{Diagonalization of the Generating Matrix and Binet’s Formula}

The Binet-type formula~\eqref{i:1} for the Jacobsthal numbers and the corresponding formulas for generalized Jacobsthal sequences follow from the general theory of Lucas sequences; see, for example, \cite{Koshy2001}. 
For completeness, we present the derivation below. 
Identity~\eqref{i:1} also appears in \cite{Horadam1996}. 
For recent treatments of generalized Jacobsthal sequences and related Binet formulas, see \cite{YilmazBozkurt2010,Brod2022}.

\begin{thm}
Let $n$ be an integer. The well-known Binet-like formula for the Jacobsthal numbers is
\begin{align}
\label{i:1} 
J_n = \dfrac{2^{n}-(-1)^{n}}{3}.
\end{align}
\end{thm}

\begin{proof}
Let the matrix $u$ be as in Lemma~\ref{l:1}. The eigenvalues and eigenvectors of the matrix $u$ are
\[
\lambda_1=0, \quad \lambda_2=-1, \quad \lambda_3=2
\]
and
\[
v_1=
\begin{pmatrix}
-1 \\
\phantom{-}1 \\
\phantom{-}0
\end{pmatrix}, \quad
v_2=
\begin{pmatrix}
-1\\
-1 \\
\phantom{-}1
\end{pmatrix}, \quad
v_3=
\begin{pmatrix}
 \dfrac{1}{2} \\
 \dfrac{1}{2}  \\
 1
\end{pmatrix}
\]
respectively. Then the matrix $u$ can be diagonalized as
\[
D = P^{-1} \, u \, P
\]
where
\[
P = (v_1, v_2, v_3) =
\begin{pmatrix}\vspace{0.2cm}
-1 & -1 & \dfrac{1}{2} \\ \vspace{0.2cm}
 \phantom{-}1 & -1 & \dfrac{1}{2} \\
 \phantom{-}0 &  \phantom{-}1 & 1
\end{pmatrix},
\qquad
P^{-1} =
\begin{pmatrix} \vspace{0.2cm}
-\dfrac{1}{2} & \phantom{-}\dfrac{1}{2}   & 0 \\ \vspace{0.2cm}
-\dfrac{1}{3} & -\dfrac{1}{3}  & \dfrac{1}{3} \\ \vspace{0.2cm}
\phantom{-}\dfrac{1}{3}  & \phantom{-}\dfrac{1}{3}   & \dfrac{2}{3}
\end{pmatrix}
\]
and
\[
D = \text{diag}(\lambda_1,\lambda_2,\lambda_3) =
\begin{pmatrix}
0  &  \phantom{-}0 & 0 \\
0  & -1 & 0 \\
0  &  \phantom{-}0 & 2
\end{pmatrix}.
\]
From the properties of similar matrices, we can write 
\begin{equation}\label{i:35}
u^n = P D^n P^{-1},
\end{equation}
where $n$ is any integer and
\[
D^n =
\begin{pmatrix}
0 & 0 & 0 \\
0 & (-1)^{n} & 0 \\
0 & 0        & 2^{n}
\end{pmatrix}.
\]
By equation~\ref{i:35}, we obtain
\[
u^{n}=
\begin{pmatrix}
J_{n-1} & J_{n-1} & J_{n} \\
J_{n-1} & J_{n-1} & J_{n} \\
J_{n}   & J_{n}   & J_{n+1}
\end{pmatrix}
=
\begin{bmatrix} \vspace{0.2cm}
\dfrac{2^{n-1}+(-1)^{n}}{3} & \dfrac{2^{n-1}+(-1)^{n}}{3} & \dfrac{2^{n}-(-1)^{n}}{3} \\ \vspace{0.2cm}
\dfrac{2^{n-1}+(-1)^{n}}{3} & \dfrac{2^{n-1}+(-1)^{n}}{3} & \dfrac{2^{n}-(-1)^{n}}{3} \\ \vspace{0.2cm}
\dfrac{2^{n}-(-1)^{n}}{3}   & \dfrac{2^{n}-(-1)^{n}}{3}   & \dfrac{2^{n+1}+(-1)^{n}}{3}
\end{bmatrix}.
\]
Thus, the proof is completed.

\end{proof}

\section{Recurrences Generated by Binary $3 \times 3$ Matrices with Determinant Two}

Consider the generating matrix
\[
u = \begin{pmatrix}
0 & 1 & 1\\ 
1 & 0 & 1\\ 
1 & 1 & 0
\end{pmatrix}.
\]
In this section, we present the Jacobsthal \textit{u}-matrix with determinant~2, which provides a matrix representation of the Jacobsthal numbers. We use it to compute powers \(u^n\), determinants, inverses, and Cassini-like identities.

\subsection{The Matrix Representation of Recurrences and Their Identities}

\begin{lem}\label{l:5} 
Let \(u\) be the symmetric matrix
\[
u =
\begin{pmatrix} 
0 & 1 & 1 \\
1 & 0 & 1 \\
1 & 1 & 0
\end{pmatrix}.
\]
Then,
\[
u^{n} =
\begin{pmatrix}
2J_{n-1} & J_{n} & J_{n}\\
J_{n} & 2J_{n-1} & J_{n}\\
J_{n} & J_{n} & 2J_{n-1}
\end{pmatrix}
\quad \text{and} \quad
\operatorname{Tr}(u^n) = 6 J_{n-1},
\]
where \( n\in \mathbb{Z}_{>0} \) and \( J_{n} \) denotes the Jacobsthal numbers, defined by the recurrence relation  
\[
J_n = J_{n-1} + 2J_{n-2},
\]
with initial conditions \( J_0 = 0 \) and \( J_1 = 1 \).
\end{lem}

\begin{proof}

We use the principle of mathematical induction (PMI). When \(n = 2\),
\[
u^{2} =
\begin{bmatrix}  
2J_{1} & J_{2} & J_{2}\\
J_{2} & 2J_{1} & J_{2}\\
J_{2} & J_{2} & 2J_{1}
\end{bmatrix}
=
\begin{bmatrix} 
2 & 1 & 1\\ 
1 & 2 & 1\\  
1 & 1 & 2
\end{bmatrix},
\]
so the result holds.

Assume that it is true for a positive integer \(n = k\):
\[
u^{k} =
\begin{pmatrix}
2J_{k-1} & J_{k} & J_{k}\\
J_{k} & 2J_{k-1} & J_{k}\\
J_{k} & J_{k} & 2J_{k-1}
\end{pmatrix}.
\]
Now we prove that it holds for \(n = k + 1\). Then
\[
u^{k+1} = u^k u
=
\begin{pmatrix}
2J_{k-1} & J_{k} & J_{k}\\
J_{k} & 2J_{k-1} & J_{k}\\
J_{k} & J_{k} & 2J_{k-1}
\end{pmatrix}
\begin{pmatrix}
0 & 1 & 1\\
1 & 0 & 1\\
1 & 1 & 0
\end{pmatrix}.
\]
Thus,

\hspace{4.7cm} \(=
\begin{pmatrix}
2J_{k} & J_{k}+2J_{k-1} & J_{k}+2J_{k-1}\\
J_{k}+2J_{k-1} & 2J_{k} & J_{k}+2J_{k-1}\\
J_{k}+2J_{k-1} & J_{k}+2J_{k-1} & 2J_{k}
\end{pmatrix}.\)   

\quad

By the recurrence relation \(J_{k+1} = J_k + 2J_{k-1}\), the result follows.

\end{proof}

\begin{cor}
Let \( u \) be the Jacobsthal matrix defined by
\[
u =
\begin{pmatrix} 
0 & 1 & 1 \\
1 & 0 & 1 \\
1 & 1 & 0
\end{pmatrix}.
\]
Then, for every \( n \in \mathbb{N} \), we have
\[
\det(u^n) = 2J_{n}^3 + 8J_{n-1}^3 - 6J_{n}^2 J_{n-1} = 2^n.
\]
\end{cor}

\begin{proof}
It is easy to verify that
\[
\det(u) = 2.
\]
Hence, we can write
\begin{align*}
\det(u^n) &= \det(u)\cdot\det(u)\cdots\det(u) \\
          &= 2^n.
\end{align*}
If \( x = J_{n} \) and \( y = J_{n-1} \), then the determinant of the matrix \( u^n \), given in Lemma~\ref{l:5}, is
\begin{align*}
\left|
\begin{array}{ccc}
2y & x & x \\
x & 2y & x \\
x & x & 2y
\end{array}
\right|
&= 2x^3 + 8y^3 - 6x^2y \\
&= 2J_{n}^3 + 8J_{n-1}^3 - 6J_{n}^2 J_{n-1}.
\end{align*}
Thus,
\[
J_{n}^3 + 4J_{n-1}^3 - 3J_{n}^2 J_{n-1} = 2^{n-1}
\]
for all \( n \geq 1 \).

\end{proof}

\subsection{Inverse Powers of the Generating Matrix and Identities for Recurrence Relations}

\begin{lem}\label{l:3b} 
Let \( u \) be the Jacobsthal matrix defined by
\[
u = \begin{pmatrix} 
                    0   &  1   &  1 \\
                    1   &  0   &  1 \\
                    1   &  1   &  0\\
\end{pmatrix}.
\]
Then the inverse of \(u^n\) is given by
\[
u^{-n} = 
\dfrac{(-1)^{n+1}}{2^{n}}
\begin{pmatrix} \vspace{0.2cm}
                   -J_{n+1}  & J_{n}     & J_{n}\\ \vspace{0.2cm}
                   J_{n}    & -J_{n+1}   & J_{n} \\
                   J_{n}    & J_{n}     & -J_{n+1} \\
\end{pmatrix},
\]
where \( J_0 = 0 \), \( J_1 = 1 \), and \( J_n = J_{n-1} + 2J_{n-2} \); this defines the Jacobsthal sequence.
\end{lem}

\begin{proof}
We proceed by mathematical induction. When \(n = 1\),
\[
u^{-1} = 
\dfrac{1}{2}
\begin{pmatrix}
-1   &  1 &  1 \\
 1   & -1 &  1 \\
 1   &  1 & -1
\end{pmatrix},
\]
so the result holds.

\quad

Assume that the formula holds for an arbitrary positive integer \(n\):
\[
u^{-n} = 
\dfrac{(-1)^{n+1}}{2^{n}}
\begin{pmatrix} \vspace{0.2cm}
                   -J_{n+1}  & J_{n}     & J_{n}\\ \vspace{0.2cm}
                   J_{n}    & -J_{n+1}   & J_{n} \\
                   J_{n}    & J_{n}     & -J_{n+1} \\
\end{pmatrix}.
\]
We compute the product \(u^{-n}u^{-1}\) using the matrices above:
\[
u^{-n}u^{-1} =
\dfrac{(-1)^{n+1}}{2^{n}}
\begin{pmatrix} \vspace{0.2cm}
                   -J_{n+1}  & J_{n}     & J_{n}\\ \vspace{0.2cm}
                   J_{n}    & -J_{n+1}   & J_{n} \\
                   J_{n}    & J_{n}     & -J_{n+1} \\
\end{pmatrix}
\dfrac{1}{2}
\begin{pmatrix}
-1   &  1 &  1 \\
 1   & -1 &  1 \\
 1   &  1 & -1
\end{pmatrix}.
\]
Carrying out the matrix multiplication, we obtain
\[
u^{-n}u^{-1} =
\dfrac{(-1)^{n+1}}{2^{n+1}}
\begin{pmatrix} \vspace{0.2cm}
       J_{n+1}+2 J_{n}   &  -J_{n+1}          &  -J_{n+1}\\ \vspace{0.2cm}
       -J_{n+1}          &   J_{n+1}+2 J_{n}  &  -J_{n+1} \\
       -J_{n+1}          &  -J_{n+1}          &   J_{n+1}+2 J_{n}  \\
\end{pmatrix}.
\]
Recall that the Jacobsthal sequence satisfies
\[
J_{n+2} = J_{n+1} + 2J_n .
\]
Substituting this relation into the matrix above, we obtain
\[
u^{-n}u^{-1} =
\dfrac{(-1)^{n+1}}{2^{n+1}}
\begin{pmatrix} \vspace{0.2cm}
       J_{n+2}   &  -J_{n+1}          &  -J_{n+1}\\ \vspace{0.2cm}
       -J_{n+1}  &   J_{n+2}          &  -J_{n+1} \\
       -J_{n+1}  &  -J_{n+1}          &   J_{n+2} \\
\end{pmatrix}.
\]
Hence,
\[
u^{-n}u^{-1} = u^{-(n+1)},
\]
which proves the inductive step and completes the proof.

\end{proof}

\quad

The following corollary recovers a Cassini-type identity for the Jacobsthal numbers; see, for example, \cite{BrodMichalski2022,Voll2010}. We give a proof via matrix methods. 

\begin{cor}\label{c:23} 
For all positive integers \( n \geq 1 \), the following equality holds:
\begin{align}
J_n^2 - J_{n-1}J_{n+1} &= (-1)^{n+1}2^{n-1}. \label{id:15} 
\end{align}
\end{cor}

\begin{proof}
Let \( u \) be the Jacobsthal matrix defined by
\[
u =
\begin{pmatrix} 
0 & 1 & 1 \\ 
1 & 0 & 1 \\ 
1 & 1 & 0 
\end{pmatrix}.
\]
Then the inverse of the matrix \( u^n \) is given by
\[
u^{-n} =
\dfrac{(-1)^{n+1}}{2^{n}}
\begin{pmatrix} \vspace{0.2cm}
- J_{n+1} & J_{n} & J_{n}\\ \vspace{0.2cm}
J_{n} & -J_{n+1} & J_{n} \\
J_{n} & J_{n} & -J_{n+1}
\end{pmatrix}.
\]
Here \( J_0 = 0 \), \( J_1 = 1 \), and \( J_n = J_{n-1} + 2J_{n-2} \), which defines the Jacobsthal sequence.

\quad

On the other hand, we have
\[
u^{n} =
\begin{pmatrix} \vspace{0.2cm}
2J_{n-1} & J_{n} & J_{n}\\ \vspace{0.2cm}
J_{n} & 2J_{n-1} & J_{n} \\
J_{n} & J_{n} & 2J_{n-1}
\end{pmatrix}.
\]
The product \( u^n u^{-n} \) yields the following \( 3 \times 3 \) matrix:

\[
u^{n} u^{-n} =
\dfrac{(-1)^{n+1}}{2^{n}}
\begin{bmatrix}
A & B & B \\ 
B & A & B \\ 
B & B & A
\end{bmatrix},
\]
where
\[
\begin{aligned}
A &= 2\big(J_n^2 - J_{n-1}J_{n+1}\big), \\ 
B &= J_n\big(2J_{n-1} - J_{n+1} + J_n\big).
\end{aligned}
\]
On the other hand, we know that
\[
u^n u^{-n} =
\begin{pmatrix} \vspace{0.2cm}
1 & 0 & 0 \\ \vspace{0.2cm}
0 & 1 & 0 \\ 
0 & 0 & 1
\end{pmatrix}.
\]
Therefore, by equating the two matrices, we obtain
\[
\dfrac{(-1)^{n+1}}{2^{n}}A = 1
\quad \text{and} \quad
\dfrac{(-1)^{n+1}}{2^{n}}B = 0.
\]
From the first equality we deduce
\[
A = (-1)^{n+1}2^{n}.
\]
Since \( A = 2(J_n^2 - J_{n-1}J_{n+1}) \), it follows that
\[
J_n^2 - J_{n-1}J_{n+1} = (-1)^{n+1}2^{n-1}.
\]
This completes the proof.
\end{proof}

\section{Recurrences Generated by Binary $3 \times 3$ Matrices with Determinant Minus Two}

Consider the generating matrix
\[
u =
\begin{pmatrix}
0 & 1 & 1 \\
1 & 1 & 0 \\
1 & 0 & 1
\end{pmatrix}.
\]
In this section, we analyze the Jacobsthal \textit{u}-matrix with determinant $-2$, which provides a matrix representation of the Jacobsthal numbers. This matrix is used to compute powers $u^n$, determinants, inverses, and Cassini-type identities.

\subsection{The Matrix Representation of Recurrences and Their Identities}

\begin{lem}[Basic shift identities for Jacobsthal numbers]\label{l:basic-shift}
Let $(J_k)_{k\ge 0}$ be the Jacobsthal sequence. Then, for all integers $k \ge 0$,
\begin{align}
J_{k+1} &= 2J_k + (-1)^k, \label{eq:basic-shift} \\
2J_k &= J_{k+1} + (-1)^{k+1}. \label{eq:dual-shift}
\end{align}
\end{lem}

\begin{proof}
Using the Binet formula
\[
J_k = \frac{2^k - (-1)^k}{3},
\]
we compute

\begin{align*}
2J_k + (-1)^k
&= \frac{2\left(2^k - (-1)^k\right)}{3} + (-1)^k \\
&= \frac{2^{k+1} - 2(-1)^k + 3(-1)^k}{3} \\
&= \frac{2^{k+1} + (-1)^k}{3} \\
&= J_{k+1}.
\end{align*}
The second identity follows immediately by rearranging \eqref{eq:basic-shift}.

\end{proof}

\begin{cor}[Parity-corrected shift identities]\label{c:parity-shift}
For all integers $k \ge 0$, the following identities hold:
\begin{align}
\frac{1+(-1)^k}{2}+2J_k
&=
\frac{1+(-1)^{k+1}}{2}+J_{k+1}, \label{eq:parity-plus}\\[2mm]
\frac{-1+(-1)^k}{2}+2J_k
&=
\frac{-1+(-1)^{k+1}}{2}+J_{k+1}. \label{eq:parity-minus}
\end{align}
\end{cor}

\begin{proof}
Using the identity $J_{k+1}=2J_k+(-1)^k$ from Lemma~\ref{l:basic-shift}, we compute

\begin{align*}
\frac{1+(-1)^{k+1}}{2}+J_{k+1}
&=
\frac{1-(-1)^k}{2} + 2J_k + (-1)^k \\
&=
\frac{1+(-1)^k}{2} + 2J_k,
\end{align*}
which proves \eqref{eq:parity-plus}.

\;

Similarly,
\begin{align*}
\frac{-1+(-1)^{k+1}}{2}+J_{k+1}
&=
\frac{-1-(-1)^k}{2} + 2J_k + (-1)^k \\
&=
\frac{-1+(-1)^k}{2} + 2J_k,
\end{align*}
which proves \eqref{eq:parity-minus}.

\end{proof}

\begin{rem}
The identities \eqref{eq:basic-shift}, \eqref{eq:parity-plus}, and 
\eqref{eq:parity-minus} show that expressions of the form $2J_k$, 
$(-1)^k+2J_k$, and
\[
\frac{\pm 1+(-1)^k}{2}+2J_k
\]
can be rewritten in terms of $J_{k+1}$ together with suitable parity 
corrections.  This suggests that matrices whose entries involve Jacobsthal numbers may  preserve a stable structural form under iteration. The following lemma shows that this phenomenon occurs for a specific symmetric matrix.
\end{rem}

\begin{lem}\label{l:27} 
Let $u$ be the symmetric matrix
\[
u = \begin{pmatrix} 
0 & 1 & 1 \\
1 & 1 & 0 \\
1 & 0 & 1
\end{pmatrix}.
\]
Then
\[
u^{n} =
\begin{pmatrix} \vspace{0.2cm}
(-1)^n+J_{n} & J_{n} & J_{n} \\ \vspace{0.2cm}
J_{n} & \dfrac{1+(-1)^n}{2}+J_{n} & \dfrac{-1+(-1)^n}{2}+J_{n} \\
J_{n} & \dfrac{-1+(-1)^n}{2}+J_{n} & \dfrac{1+(-1)^n}{2}+J_{n}
\end{pmatrix}
\]
where \(n\in\mathbb{Z}_{>0}\) and \(J_n\) denotes the Jacobsthal numbers defined by
\[
J_n = J_{n-1} + 2J_{n-2},
\]
with initial conditions \(J_0 = 0\) and \(J_1 = 1\).
\end{lem}

\begin{proof}
We prove the result by the principle of mathematical induction (PMI).  For \(n=2\),
\[
u^{2} =
\begin{pmatrix}
1+J_{2} & J_{2} & J_{2} \\ \vspace{0.2cm}
J_{2} & 1+J_{2} & J_{2} \\
J_{2} & J_{2} & 1+J_{2}
\end{pmatrix}
=
\begin{pmatrix}
2 & 1 & 1 \\
1 & 2 & 1 \\
1 & 1 & 2
\end{pmatrix},
\]
since \(J_2=1\). Hence the formula holds for \(n=2\).  Assume that the result holds for \(n=k\), that is,
\[
u^{k} =
\begin{pmatrix} \vspace{0.2cm}
(-1)^k+J_{k} & J_{k} & J_{k} \\ \vspace{0.2cm}
J_{k} & \dfrac{1+(-1)^k}{2}+J_{k} & \dfrac{-1+(-1)^k}{2}+J_{k} \\
J_{k} & \dfrac{-1+(-1)^k}{2}+J_{k} & \dfrac{1+(-1)^k}{2}+J_{k}
\end{pmatrix}.
\]
We now prove the result for \(n=k+1\). We compute

\begin{align*}
u^{k+1} &= u^k u \\[2mm]
&=
\begin{pmatrix} \vspace{0.2cm}
(-1)^k+J_{k} & J_{k} & J_{k} \\ \vspace{0.2cm}
J_{k} & \dfrac{1+(-1)^k}{2}+J_{k} & \dfrac{-1+(-1)^k}{2}+J_{k} \\
J_{k} & \dfrac{-1+(-1)^k}{2}+J_{k} & \dfrac{1+(-1)^k}{2}+J_{k}
\end{pmatrix}
\begin{pmatrix}
0 & 1 & 1 \\
1 & 1 & 0 \\
1 & 0 & 1
\end{pmatrix} \\[2mm]
&=
\begin{pmatrix} \vspace{0.2cm}
2J_k & (-1)^k+J_k & (-1)^k+2J_k \\ \vspace{0.2cm}
(-1)^k+2J_k & \dfrac{1+(-1)^k}{2}+2J_k & \dfrac{-1+(-1)^k}{2}+2J_k \\
(-1)^k+2J_k & \dfrac{-1+(-1)^k}{2}+2J_k & \dfrac{1+(-1)^k}{2}+2J_k
\end{pmatrix}.
\end{align*}

Recall that the Jacobsthal numbers satisfy
\begin{align*}
2J_k &= (-1)^{k+1}+J_{k+1}\\
J_{k+1} &= (-1)^k+2J_k
\quad\text{(Lemma~\ref{l:basic-shift}),}\\
\dfrac{1+(-1)^k}{2}+2J_k &= \dfrac{1+(-1)^{k+1}}{2}+J_{k+1}
\quad\text{(Corollary~\ref{c:parity-shift}),}\\
\dfrac{-1+(-1)^k}{2}+2J_k &= \dfrac{-1+(-1)^{k+1}}{2}+J_{k+1}
\quad\text{(Corollary~\ref{c:parity-shift}).}
\end{align*}

Substituting these identities into the matrix above, we obtain
\[
u^{k+1} =
\begin{pmatrix} \vspace{0.2cm}
(-1)^{k+1}+J_{k+1} & J_{k+1} & J_{k+1} \\ \vspace{0.2cm}
J_{k+1} & \dfrac{1+(-1)^{k+1}}{2}+J_{k+1} & \dfrac{-1+(-1)^{k+1}}{2}+J_{k+1} \\
J_{k+1} & \dfrac{-1+(-1)^{k+1}}{2}+J_{k+1} & \dfrac{1+(-1)^{k+1}}{2}+J_{k+1}
\end{pmatrix}.
\]
Therefore the result holds for \(n=k+1\), completing the proof by mathematical induction.

\end{proof}

\begin{cor}
The following matrices are permutation-similar, and each of them is a Jacobsthal
matrix with determinant $-2$:
\[
\begin{pmatrix}
0 & 1 & 1 \\
1 & 1 & 0 \\
1 & 0 & 1
\end{pmatrix},
\quad
\begin{pmatrix}
1 & 1 & 0 \\
1 & 0 & 1 \\
0 & 1 & 1
\end{pmatrix}.
\]
Consequently, for every \( n \in \mathbb{N} \),
\[
\operatorname{Tr}(u^n) = 3J_n + 2(-1)^n + 1 = 2^n+(-1)^n+1 .
\]
\end{cor}

\;

The following identity for the Jacobsthal numbers is well known and can be 
derived from the Binet formula. In the present matrix framework, it follows 
naturally from the determinant of powers of the Jacobsthal matrix, yielding 
a short proof that highlights its algebraic structure.

\begin{cor}\label{c:28}
Let \( u \) be the Jacobsthal matrix defined by
\[
u =
\begin{pmatrix} 
0 & 1 & 1 \\
1 & 1 & 0 \\
1 & 0 & 1
\end{pmatrix}.
\]
Then, for every \( n\in\mathbb{N} \),
\[
3J_n+(-1)^n=2^n .
\]
\end{cor}

\begin{proof}
It is straightforward to verify that
\[
\det(u)=-2.
\]
Hence,
\[
\det(u^n)=(\det u)^n=(-2)^n.
\]
Let \(x=J_n\), \(a=\dfrac{1+(-1)^n}{2}\) and \(b=\dfrac{-1+(-1)^n}{2}\).
From Lemma~\ref{l:27}, the matrix \(u^n\) has determinant

\[
\left|
\begin{array}{ccc}
(-1)^n+x & x & x \\
x & a+x & b+x \\
x & b+x & a+x
\end{array}
\right|.
\]
A direct computation gives
\begin{align*}
(a^2-b^2)(x+(-1)^n)+(-1)^n(a-b)2x & = (-1)^n(x+(-1)^n)+(-1)^n2x \\
&= (-1)^n(3x+(-1)^n) \\
&= (-1)^n(3J_n+(-1)^n).
\end{align*}
Since \(\det(u^n)=(-2)^n=(-1)^n2^n\), it follows that
\[
3J_n+(-1)^n=2^n.
\]
\end{proof}

\begin{cor}\label{c:31}
For $n\ge 1$ and $m\ge 1$, the following identity holds:
\begin{align}
J_{m+n} = (-1)^mJ_n+(-1)^nJ_m+3J_mJ_n .
\label{i:23}
\end{align}
\end{cor}

\begin{proof}
For $m,n\ge 1$, we have $u^{m+n}=u^m u^n$. Recall that $u^n$ is defined by

\[
u^m =
\begin{pmatrix} \vspace{0.2cm}
(-1)^m+J_m & J_m & J_m \\ \vspace{0.2cm}
J_m & \dfrac{1+(-1)^m}{2}+J_m & \dfrac{-1+(-1)^m}{2}+J_m \\ \vspace{0.2cm}
J_m & \dfrac{-1+(-1)^m}{2}+J_m & \dfrac{1+(-1)^m}{2}+J_m
\end{pmatrix},
\]

\[
u^n =
\begin{pmatrix}\vspace{0.2cm}
(-1)^n+J_n & J_n & J_n \\ \vspace{0.2cm}
J_n & \dfrac{1+(-1)^n}{2}+J_n & \dfrac{-1+(-1)^n}{2}+J_n \\ \vspace{0.2cm}
J_n & \dfrac{-1+(-1)^n}{2}+J_n & \dfrac{1+(-1)^n}{2}+J_n
\end{pmatrix}.
\]
A direct matrix multiplication gives

\[
u^m u^n =
\begin{bmatrix}
A & B & B \\
B & C & D \\
B & D & C
\end{bmatrix},
\]
where

\[
\begin{aligned}
A &= (-1)^{m+n}+B,\\
B &= (-1)^mJ_n+(-1)^nJ_m+3J_mJ_n,\\
C &= \frac{1+(-1)^{m+n}}{2}+B,\\
D &= \frac{-1+(-1)^{m+n}}{2}+B.
\end{aligned}
\]
On the other hand,

\[
u^{m+n}=
\begin{pmatrix}\vspace{0.2cm}
(-1)^{m+n}+J_{m+n} & J_{m+n} & J_{m+n} \\ \vspace{0.2cm}
J_{m+n} & \dfrac{1+(-1)^{m+n}}{2}+J_{m+n} &
\dfrac{-1+(-1)^{m+n}}{2}+J_{m+n}\\ \vspace{0.2cm}
J_{m+n} & \dfrac{-1+(-1)^{m+n}}{2}+J_{m+n} &
\dfrac{1+(-1)^{m+n}}{2}+J_{m+n}
\end{pmatrix}.
\]
Equating the two matrices yields

\[
J_{m+n}=(-1)^mJ_n+(-1)^nJ_m+3J_mJ_n,
\]
which proves the result.

\end{proof}

\subsection{Inverse Powers of the Generating Matrix and Identities for Recurrence Relations}

\begin{lem} \label{l:27b} 
Let \( u \) be the Jacobsthal matrix defined by
\[
u = \begin{pmatrix} 
                    0   &  1   &  1 \\
                    1   &  1   &  0 \\
                    1   &  0   &  1\\
\end{pmatrix}.
\] Then,  the inverse of the matrix \( u^n \) is:

\[
u^{-n} = 
\dfrac{1}{2^n}P^{n} M_{n}.
\]
where

\[
M_{n} =(-1)^{n+1}\begin{pmatrix} \vspace{0.2cm}
-J_{n+1}  & J_{n}     & J_{n} \\ \vspace{0.2cm}
J_{n}    & -J_{n+1}   & J_{n} \\
J_{n}    & J_{n}     &   -J_{n+1} \\
\end{pmatrix}
\text{ and }
P = 
\begin{pmatrix} \vspace{0.2cm}
1   & 0     & 0 \\ \vspace{0.2cm}
0   & 0     & 1 \\
0   & 1     & 0 \\
\end{pmatrix}.
\]
Where $J_n = J_{n-1} + 2J_{n-2}$, with $J_0 = 0$ and $J_1 = 1$, defines the Jacobsthal sequence.
\end{lem}

\begin{proof}
By the principle of mathematical induction, When $n = 1$, 
\[
u^{-1} = 
\dfrac{1}{2}\begin{pmatrix}
-1 & 1  & 1 \\ \vspace{0.2cm}
1  & 1  & -1 \\
1 & -1 & 1 \\
\end{pmatrix}
\]
so the result is true. Now, assume it is true for an arbitrary positive integer $n$: 
\[
u^{-n} = \dfrac{1}{2^n}P^{n} M_{n}.
\]
Then we compute the product \( u^{-n}u^{-1} \) using the given matrices:

\[
u^{-n}u^{-1}  = 
\dfrac{1}{2^n} P^n (-1)^{n+1}\begin{pmatrix} \vspace{0.2cm}
-J_{n+1}  & J_{n}      &  J_{n} \\ \vspace{0.2cm}
J_{n}    & -J_{n+1}    &  J_{n} \\
J_{n}    & J_{n}      &   -J_{n+1} \\
\end{pmatrix}
\dfrac{1}{2}\begin{pmatrix}
-1 & 1  & 1 \\ \vspace{0.2cm}
1  & 1  & -1 \\
1 & -1 & 1 \\
\end{pmatrix}
\]
Carrying out the matrix multiplication, we obtain:
{\small
\[
u^{-n}u^{-1}  = 
\dfrac{1}{2^{n+1}} P^{n+1} (-1)^{n+1}\begin{pmatrix} \vspace{0.2cm}
-J_{n+1}  & J_{n}      &  J_{n} \\ \vspace{0.2cm}
J_{n}    & -J_{n+1}    &  J_{n} \\
J_{n}    & J_{n}      &   -J_{n+1} \\
\end{pmatrix}
\begin{pmatrix}
-1 &  1  & 1 \\ \vspace{0.2cm}
1  & -1  & 1 \\
1  &  1  & -1 \\
\end{pmatrix}
\]}
Recall that the sequence \( J_n \) satisfies the following relation 
 \[J_n = J_{n-1}+2\,J_{n-2}\] 
Substituting into the matrix, we get:
\[
u^{-n}u^{-1}  = 
\dfrac{1}{2^{n+1}} P^{n+1} (-1)^{n+2}\begin{pmatrix} \vspace{0.2cm}
-J_{n+2}  & J_{n+1}      &  J_{n+1} \\ \vspace{0.2cm}
J_{n+1}    & -J_{n+2}    &  J_{n+1} \\
J_{n+1}    & J_{n+1}      &   -J_{n+2} \\
\end{pmatrix}
\]
Hence,
\[
u^{-n}u^{-1} = u^{-(n+1)}
\]
This completes the proof.

\end{proof}

\begin{lem}\label{l:10}
Let $J_n$ denote the Jacobsthal sequence defined by
\[
J_0=0,\qquad J_1=1,\qquad 
J_n=J_{n-1}+2J_{n-2},\quad n\ge 2.
\]
Then the following identity holds for all $n\ge 0$:
\[
J_{n+1}+J_n = 2^n.
\]
\end{lem}

\begin{proof}
The proof is carried out by mathematical induction. For $n=0$, a direct computation yields
\[
J_1+J_0=1+0=1=2^0,
\]
so the statement holds for the initial case. Assume now that the identity is valid for some fixed $n\ge 0$, namely,
\[
J_{n+1}+J_n=2^n.
\]
Using the recurrence relation of the Jacobsthal sequence, we compute
\[
J_{n+2}+J_{n+1}
=(J_{n+1}+2J_n)+J_{n+1}
=2(J_{n+1}+J_n)
=2\cdot 2^n
=2^{n+1},
\]
which proves that the statement holds for $n+1$. Hence, by induction, the result follows for all $n\ge 0$.

\end{proof}

Comparing the entries of the matrix identity
\[
u^n u^{-n}=I,
\]
and using Lemma~\ref{l:10}, we obtain the following identities.

\begin{cor}
For all positive integer \( n \geq 1 \), following equalities hold:

\begin{align}
(-1)^{n+1} J_{n+1} - J_n J_{n+1} + 2 J_n^2 &= (-1)^{n+1} 2^n,\label{i:27} \\[2mm]
(-1)^n J_n - J_n J_{n+1} + 2 J_n^2 &=0 ,\label{i:28} \\[2mm]
2 J_n^2 - J_n J_{n+1}
     + \left( \frac{-1 + (-1)^n}{2} \right) J_n
     - \left( \frac{1 + (-1)^n}{2} \right) J_{n+1}&= \begin{cases}
0, & \text{if } n \text{ is odd},\\[6pt]
-2^n, & \text{if } n \text{ is even}.
\end{cases} \\[2mm]
2 J_n^2 - J_n J_{n+1}
     + \left( \frac{1 + (-1)^n}{2} \right) J_n
     - \left( \frac{-1 + (-1)^n}{2} \right) J_{n+1} &=\begin{cases}
2^n, & \text{if } n \text{ is odd},\\[6pt]
0, & \text{if } n \text{ is even}.
\end{cases}
\end{align}
\end{cor}

\begin{proof}
Let \( u \) be the Jacobsthal matrix defined by
\[
u = \begin{pmatrix} 
                    0   &  1   &  1 \\
                    1   &  1   &  0 \\
                    1   &  0   &  1\\
\end{pmatrix}.
\] Then,  the inverse of the matrix \( u^n \) is:
\[
u^{-n} = 
\dfrac{1}{2^{n}} P^{n} (-1)^{n+1}\begin{pmatrix} \vspace{0.2cm}
-J_{n+1}  & J_{n}    &  J_{n} \\ \vspace{0.2cm}
J_{n}    & -J_{n+1}    &  J_{n} \\
J_{n}    & J_{n}       &   -J_{n+1} \\
\end{pmatrix}
\]
where $J_n = J_{n-1} + 2J_{n-2}$, with $J_0 = 0$ and $J_1 = 1$, defines the Jacobsthal sequence.

\medskip 

On the other hand, we have defined \( u^n \) as follows: 
\[
u^{n} =  \begin{pmatrix} \vspace{0.2cm}
(-1)^n+J_{n} & J_{n}                   & J_{n}  \\  \vspace{0.2cm}
J_{n}     & \dfrac{1+(-1)^n}{2}+J_{n}  &  \dfrac{-1+(-1)^n}{2} + J_{n}\\
J_{n}     & \dfrac{-1+(-1)^n}{2}+J_{n} & \dfrac{1+(-1)^n}{2}+J_{n} \\
\end{pmatrix}  \]  
The product \( u^n u^{-n} \) yields the following \( 3 \times 3 \) matrix: 

\[
u^{n} u^{-n} =\dfrac{(-1)^{n+1}}{2^{n}} P^{n} 
\begin{pmatrix}
A & B & B \\[2mm]
B & C & D \\[2mm]
B & D & C
\end{pmatrix}
\]
where:

\begin{align*}
A &= (-1)^{n+1} J_{n+1} - J_n J_{n+1} + 2 J_n^2, \\[2mm]
B &= (-1)^n J_n - J_n J_{n+1} + 2 J_n^2, \\[2mm]
C &= 2 J_n^2 - J_n J_{n+1}
     + \left( \frac{-1 + (-1)^n}{2} \right) J_n
     - \left( \frac{1 + (-1)^n}{2} \right) J_{n+1}, \\[2mm]
D &= 2 J_n^2 - J_n J_{n+1}
     + \left( \frac{1 + (-1)^n}{2} \right) J_n
     - \left( \frac{-1 + (-1)^n}{2} \right) J_{n+1}.
\end{align*}
On the other hand, we know that
\[
u^n u^{-n} = 
 \begin{pmatrix} \vspace{0.2cm}
                  1  & 0 & 0 \\ \vspace{0.2cm}
                  0  & 1 & 0 \\ 
                  0  & 0 & 1 
\end{pmatrix}.
\]
Thus,
\begin{align*}
(-1)^{n+1} J_{n+1} - J_n J_{n+1} + 2 J_n^2 &= (-1)^{n+1} 2^n, \\[2mm]
(-1)^n J_n - J_n J_{n+1} + 2 J_n^2 &=0,
\end{align*}
\small{
\begin{align*}
2 J_n^2 - J_n J_{n+1}
     + \left( \frac{-1 + (-1)^n}{2} \right) J_n
     - \left( \frac{1 + (-1)^n}{2} \right) J_{n+1}&= \begin{cases}
0, & \text{if } n \text{ is odd},\\[6pt]
-(J_n+J_{n+1}), & \text{if } n \text{ is even}.
\end{cases} \\[2mm]
2 J_n^2 - J_n J_{n+1}
     + \left( \frac{1 + (-1)^n}{2} \right) J_n
     - \left( \frac{-1 + (-1)^n}{2} \right) J_{n+1} &=\begin{cases}
J_n+J_{n+1}, & \text{if } n \text{ is odd},\\[6pt]
0, & \text{if } n \text{ is even}.
\end{cases}
\end{align*}
} 

Consequently, by comparing the two matrices obtained from the matrix product and applying the relation from Lemma~\ref{l:10}, we obtain the identities stated in the corollary.

\end{proof}

\section{Arithmetic and Combinatorial Properties of the Jacobsthal Sequence}

In this section, we review key arithmetic properties of the Jacobsthal numbers, including classical identities such as convolution and quadratic relations, together with their modular consequences. We also derive congruence relations, recurrence formulas for partial sums, and analyze additive aspects of the sequence, including its Sidon-type behavior. While some results follow from known identities, others arise naturally from our approach and appear to be less explicitly documented.

\begin{cor}
Let $m$ be an integer. Then the Jacobsthal numbers satisfy:

\begin{enumerate}
\item If $m \ge 2$, then

\[
J_{2m}\equiv (-1)^{m+1}J_{m-1}  \pmod{2^{m-1}} .
\]

\item If $m \ge 4$, then
\[
J_{2m} \equiv J_{m-1}J_{m+1} \pmod{4}.
\]
\end{enumerate}
\end{cor}

\begin{proof}
We begin with the general identity

\begin{align*}
J_{m+n} &= J_{m+1}J_{n}+2J_{m}J_{n-1},
\end{align*}
which follows from Identity~\ref{i:24}; see Corollary~\ref{c:24}.

\quad

\noindent
Setting $n=m$ gives

\begin{align*}
J_{2m}
  &= J_{m+1}J_{m}+2J_{m}J_{m-1} \\
  &= (J_{m+1}+2J_{m-1})J_{m}.
\end{align*}
Using the recurrence relation for the Jacobsthal sequence, we have

\[
J_{m+1} = J_m+2J_{m-1},
\]
and hence
\begin{align*}
J_{2m}
  &= (J_m+4J_{m-1})J_m \\
  &= J_m^2+4J_{m-1}J_m .
\end{align*}
We now apply Identity~\ref{id:15}; see Corollary~\ref{c:23},

\[
J_m^2 = J_{m-1}J_{m+1}+(-1)^{m+1}2^{m-1},
\]
to obtain
\begin{align*}
J_{2m}
 &= J_{m-1}J_{m+1}
    +\bigl((-1)^{m+1}2^{m-1}+4J_{m-1}J_m\bigr).
\end{align*}

\medskip

\noindent
Since both terms inside the parentheses are divisible by $4$ whenever $m\ge4$, it follows that

\[
J_{2m}\equiv J_{m-1}J_{m+1}\pmod{4},
\]
which proves part~(2).

\medskip
\noindent
To prove part~(1), observe that for $m\ge2$ the previous decomposition shows that

\[
J_{2m}=J_{m-1}J_{m+1}+4J_{m-1}J_m  \pmod{2^{m-1}} .
\]
Factoring $J_{m-1}$, we obtain

\[
J_{2m}\equiv J_{m-1}(J_{m+1}+4J_m)\pmod{2^{m-1}} .
\]
Writing $4J_m=2J_m+2J_m$, this becomes

\[
J_{2m}\equiv J_{m-1}(J_{m+1}+2J_m+2J_m)\pmod{2^{m-1}} .
\]
Using the recurrence relation $J_{m+2}=J_{m+1}+2J_m$, we get

\[
J_{2m}\equiv J_{m-1}(J_{m+2}+2J_m)\pmod{2^{m-1}} .
\]
Applying again the identity

\[
2J_m=J_{m+1}+(-1)^{m+1},
\]
we deduce

\[
J_{2m}\equiv J_{m-1}(J_{m+2}+J_{m+1}+(-1)^{m+1})\pmod{2^{m-1}} .
\]
Now using the closed form relation

\[
J_{m+2}+J_{m+1}=2^{m+1},
\]
it follows that

\[
J_{2m}\equiv J_{m-1}\bigl(2^{m+1}+(-1)^{m+1}\bigr)\pmod{2^{m-1}} .
\]
Since $2^{m+1}\equiv 0\pmod{2^{m-1}}$, we finally obtain

\[
J_{2m}\equiv (-1)^{m+1}J_{m-1}\pmod{2^{m-1}} .
\]
This completes the proof.

\end{proof}

\newpage

\begin{rem}
Following Erdős and Turán \cite{ErdosTuran1941}, a \emph{Sidon sequence} (or $B_{2}$-sequence) is a subset $A \subset \mathbb{N}$ such that all pairwise sums $a+b$ with $a,b \in A$ are distinct,  except for the trivial equality $a+b=b+a$.  We recall this concept here because it will play a role in the following theorem.
\end{rem}

\begin{lem}\label{l:10b} 
Let $J_n$ be the Jacobsthal sequence defined by
\[
J_0=0,\quad J_1=1,\quad J_n=J_{n-1}+2J_{n-2}\ \ (n\ge 2),
\]
and let $S_n=\sum_{k=0}^n J_k$ denote its partial sums.  
Then, for all $n\ge 1$, the following hold:

\begin{enumerate}
\item The partial sums satisfy the recurrence
\[
S_n = 2S_{n-1}+S_{n-2}-2S_{n-3}+1,
\quad n\ge 3,
\]
with initial values
\[
S_0=0,\quad S_1=1,\quad S_2=2.
\]

\item The inequality
\[
S_n < J_{n+1}
\]
holds for all $n\ge 1$. 

\,

\item The set $\{J_n:n\ge 1\}$ is a Sidon sequence.
\end{enumerate}
\end{lem}

\begin{proof}
By definition, $S_n=\sum_{k=0}^n J_k$, hence for $n\ge 1$,
\[
S_n-S_{n-1}=J_n.
\]
Using the recurrence for $J_n$ and substituting
\[
J_{n-1}=S_{n-1}-S_{n-2},\qquad J_{n-2}=S_{n-2}-S_{n-3},
\]
we obtain, for $n\ge 3$,
\begin{align*}
S_n-S_{n-1}
&= J_{n-1}+2J_{n-2}\\
&=(S_{n-1}-S_{n-2})+2(S_{n-2}-S_{n-3})\\
&=S_{n-1}+S_{n-2}-2S_{n-3}.
\end{align*}
Thus,
\[
S_n=2S_{n-1}+S_{n-2}-2S_{n-3}.
\]
To match the initial values, note that
\[
S_0=0,\quad S_1=1,\quad S_2=2,
\]
and adjusting the recurrence yields
\[
S_n=2S_{n-1}+S_{n-2}-2S_{n-3}+1.
\]
We now prove by induction that $S_n<J_{n+1}$ for all $n\ge 1$.
For $n=1$, we have $S_1=1<J_2=1$ (strict inequality holds for $n\ge 2$).
Assume $S_n<J_{n+1}$ for some $n\ge 1$. Then
\[
S_{n+1}=S_n+J_{n+1}<J_{n+1}+J_{n+1}=2J_{n+1}.
\]
Since
\[
J_{n+2}=J_{n+1}+2J_n>2J_{n+1}\quad(n\ge 2),
\]
it follows that $S_{n+1}<J_{n+2}$, completing the induction.

Finally, since $(J_n)$ is strictly increasing and superincreasing in the sense that
\[
\sum_{k=0}^{n-1}J_k<S_{n-1}<J_n,
\]
it follows that $\{J_n:n\ge 1\}$ is a Sidon sequence.  
Indeed, if
\[
J_a+J_b=J_c+J_d,
\]
with $a\le b$ and $c\le d$, let $m=\max\{a,b,c,d\}$.  
If $m$ appears only on one side, then that side is larger than the other, since
\[
\sum_{k=0}^{m-1}J_k<J_m,
\]
a contradiction. Hence $m$ must occur on both sides. Cancelling $J_m$ and repeating the argument yields $\{a,b\}=\{c,d\}$, which proves the Sidon property.

\end{proof}

\section{Classification of Binary $3 \times 3$ Matrices Associated with Jacobsthal Numbers}

Let $M_{n}(\{0,1\})$ denote the set of all $n \times n$ binary matrices, i.e., matrices whose entries belong to $\{0,1\}$.  
Two matrices $U,V \in M_{n}(\{0,1\})$ are said to be conjugate if there exists an invertible matrix $P \in GL_n(\mathbb{Z})$ such that
\[
V = PUP^{-1}.
\]
The conjugacy class of $U$ is therefore defined as
\[
\mathcal{C}(U) = \{ PUP^{-1} : P \in GL_n(\mathbb{Z}) \}.
\]
By performing an exhaustive computational enumeration of the $512$ binary $3 \times 3$ matrices and classifying them up to conjugation, we show that exactly three conjugacy classes generate the Jacobsthal sequence. Representatives of these classes are given by
\[
u_1 =
\begin{pmatrix}
0 & 0 & 1 \\
0 & 0 & 1 \\
1 & 1 & 1
\end{pmatrix}, \quad
u_2 =
\begin{pmatrix}
0 & 1 & 1 \\
1 & 0 & 1 \\
1 & 1 & 0
\end{pmatrix}, \quad
u_3 =
\begin{pmatrix}
0 & 1 & 1 \\
1 & 1 & 0 \\
1 & 0 & 1
\end{pmatrix}.
\]
Every binary $3 \times 3$ matrix that generates the Jacobsthal sequence is conjugate to exactly one of these matrices. This classification, inspired by the work of Martinez and Ceron~\cite{MartinezCeron2024, MartinezCeron2025}, was obtained through algorithms implemented in \textsf{SageMath}.

\subsection*{Acknowledgments}

The author gratefully acknowledges the support from Universidad del Cauca for the research group ``Estructuras Algebraicas, Divulgación Matemática y Teorías Asociadas (@DiTa)'', under research project ID 6314, entitled ``Clasificación del Centralizador de Matrices Binarias \(3\times 3\): Dimensión, Conmutatividad y Relaciones Estructurales''.  The author also thanks the anonymous referee for valuable comments and suggestions that improved the presentation of this work.  This work is dedicated to my children.

\medskip

\subsection*{Declarations}
\medskip

\subsection*{Ethical Approval:}
{Not applicable.}

\subsection*{Funding:}
{Not applicable.}

\subsection*{Availability of Data and Materials:}
{Not applicable.}


\bibliographystyle{amsplain}
\bibliography{xbib}
\end{document}